\documentclass[12pt]{amsart}
\usepackage{mathrsfs}      
\usepackage{txfonts}
\usepackage{amssymb}
\usepackage{eucal}
\usepackage{graphicx}
\usepackage{amsmath}
\usepackage{amscd}
\usepackage[all]{xy}           
\usepackage[active]{srcltx} 

\usepackage{cite}

\usepackage{pmat}
\usepackage{amscd}
\usepackage[all]{xy}           
\usepackage[active]{srcltx} 

\usepackage{amsfonts,latexsym}
\usepackage{xspace}
\usepackage{epsfig}
\usepackage{float}
\usepackage{color}
\usepackage{fancybox}
\usepackage{colordvi}
\usepackage{multicol}
\usepackage{colordvi}

\numberwithin{equation}{section}

\theoremstyle{definition}

\theoremstyle{theorem}
\newtheorem{theorem}{Theorem}[section]
\newtheorem{prop}[theorem]{Proposition}
\newtheorem{lemma}[theorem]{Lemma}
\newtheorem{coro}[theorem]{Corollary}

\newtheorem{prop-def}{Proposition-Definition}
\newtheorem{coro-def}{Corollary-Definition}[section]
\theoremstyle{definition}
\newtheorem{defn}[theorem]{Definition}
\newtheorem{remark}[theorem]{Remark}
\newtheorem{exam}[theorem]{Example}

\newcommand{\nc}{\newcommand}
\nc{\on}{\operatorname}

\nc{\tr}{\rm Tr}
\nc{\spa}{\rm span}
\nc{\ad}{{\rm ad}}
\nc{\dep}{{\rm dep}}
\nc{\diag}{{\rm diag}}

\nc{\Id}{{\rm Id}}

\nc{\Ixq}{I_{P}}
\nc{\tred}[1]{\textcolor{red}{#1}}
\nc{\tblue}[1]{\textcolor{blue}{#1}}
\nc{\tgreen}[1]{\textcolor{green}{#1}}
\nc{\tpurple}[1]{\textcolor{purple}{#1}}
\nc{\btred}[1]{\textcolor{red}{\bf #1}}
\nc{\btblue}[1]{\textcolor{blue}{\bf #1}}
\nc{\btgreen}[1]{\textcolor{green}{\bf #1}}
\nc{\btpurple}[1]{\textcolor{purple}{\bf #1}}

\newcommand{\efootnote}[1]{}

\renewcommand{\textbf}[1]{}

\newcommand{\delete}[1]{}
\nc{\dfootnote}[1]{{}}          
\nc{\ffootnote}[1]{\dfootnote{#1}}
\nc{\mfootnote}[1]{\footnote{#1}} 
\nc{\ofootnote}[1]{\footnote{\tiny Older version: #1}}

\delete{
\nc{\mlabel}[1]{\label{#1}}  
\nc{\mcite}[1]{\cite{#1}}  
\nc{\mref}[1]{\ref{#1}}  
\nc{\mcite}[1]{\cite{#1}{{\bf{{\ }(#1)}}}}  
}

\nc{\mlabel}[1]{\label{#1}  
{\hfill \hspace{1cm}{\bf{{\ }\hfill(#1)}}}}
\nc{\mcite}[1]{\cite{#1}{{\bf{{\ }(#1)}}}}  
\nc{\mref}[1]{\ref{#1}{{\bf{{\ }(#1)}}}}  

\nc{\mbibitem}[1]{\bibitem[\bf #1]{#1}} 
\nc{\mkeep}[1]{\marginpar{{\bf #1}}} 

\nc{\repmap}{representative~map~}

\nc{\opa}{\ast} \nc{\opb}{\odot} \nc{\op}{\bullet} \nc{\pa}{\frakL}
\nc{\arr}{\rightarrow} \nc{\lu}[1]{(#1)}
\nc{\opc}{\sharp}\nc{\opd}{\natural}
\nc{\ope}{\circ}
\nc{\Mod}{\operatorname{-Mod}}
\nc{\bin}[2]{ (_{\stackrel{\scs{#1}}{\scs{#2}}})}  
\nc{\binc}[2]{ \left (\!\! \begin{array}{c} \scs{#1}\\
    \scs{#2} \end{array}\!\! \right )}  
\nc{\bincc}[2]{  \left ( {\scs{#1} \atop
    \vspace{-1cm}\scs{#2}} \right )}  
\nc{\bs}{\bar{S}} \nc{\cosum}{\sqsubset} \nc{\la}{\longrightarrow}
\nc{\rar}{\rightarrow} \nc{\dar}{\downarrow} \nc{\dprod}{**}
\nc{\dap}[1]{\downarrow \rlap{$\scriptstyle{#1}$}}
\nc{\md}{\mathrm{dth}} \nc{\uap}[1]{\uparrow
\rlap{$\scriptstyle{#1}$}} \nc{\defeq}{\stackrel{\rm def}{=}}
\nc{\disp}[1]{\displaystyle{#1}} \nc{\dotcup}{\
\displaystyle{\bigcup^\bullet}\ } \nc{\gzeta}{\bar{\zeta}}
\nc{\hcm}{\ \hat{,}\ } \nc{\hts}{\hat{\otimes}}
\nc{\barot}{{\otimes}} \nc{\free}[1]{\bar{#1}}
\nc{\uni}[1]{\tilde{#1}} \nc{\hcirc}{\hat{\circ}} \nc{\lleft}{[}
\nc{\lright}{]} \nc{\lc}{\lfloor} \nc{\rc}{\rfloor}
\nc{\curlyl}{\left \{ \begin{array}{c} {} \\ {} \end{array}
    \right .  \!\!\!\!\!\!\!}
\nc{\curlyr}{ \!\!\!\!\!\!\!
    \left . \begin{array}{c} {} \\ {} \end{array}
    \right \} }
\nc{\longmid}{\left | \begin{array}{c} {} \\ {} \end{array}
    \right . \!\!\!\!\!\!\!}
\nc{\onetree}{\bullet} \nc{\ora}[1]{\stackrel{#1}{\rar}}
\nc{\ola}[1]{\stackrel{#1}{\la}}
\nc{\ot}{\otimes} \nc{\mot}{{{\boxtimes\,}}}
\nc{\otm}{\overline{\boxtimes}} \nc{\sprod}{\bullet}
\nc{\scs}[1]{\scriptstyle{#1}} \nc{\mrm}[1]{{\rm #1}}
\nc{\margin}[1]{\marginpar{\rm #1}}   
\nc{\dirlim}{\displaystyle{\lim_{\longrightarrow}}\,}
\nc{\invlim}{\displaystyle{\lim_{\longleftarrow}}\,}
\nc{\mvp}{\vspace{0.3cm}} \nc{\tk}{^{(k)}} \nc{\tp}{^\prime}
\nc{\ttp}{^{\prime\prime}} \nc{\svp}{\vspace{2cm}}
\nc{\vp}{\vspace{8cm}} \nc{\proofbegin}{\noindent{\bf Proof: }}
\nc{\proofend}{$\blacksquare$ \vspace{0.3cm}}
\nc{\modg}[1]{\!<\!\!{#1}\!\!>}
\nc{\intg}[1]{F_C(#1)} \nc{\lmodg}{\!
<\!\!} \nc{\rmodg}{\!\!>\!}
\nc{\cpi}{\widehat{\Pi}}
\nc{\sha}{{\mbox{\cyr X}}}  
\nc{\shap}{{\mbox{\cyrs X}}} 
\nc{\shpr}{\diamond}    
\nc{\shp}{\ast} \nc{\shplus}{\shpr^+}
\nc{\shprc}{\shpr_c}    
\nc{\msh}{\ast} \nc{\zprod}{m_0} \nc{\oprod}{m_1}
\nc{\vep}{\varepsilon} \nc{\labs}{\mid\!} \nc{\rabs}{\!\mid}
\nc{\mmbox}[1]{\mbox{\ #1\ }}

\nc{\rsd}{\operatorname{RSD}}
\nc{\Ker}{\operatorname{Ker}}
\nc{\mult}{\operatorname{mult}}
\nc{\diff}{\mathfrak{Diff}}
\nc{\rank}{\operatorname{Rank}}
\nc{\stab}{\operatorname{Stab}}
\nc{\aut}{\operatorname{Aut}}
 \nc{\fp}{\operatorname{FP}}
\nc{\rchar}{\operatorname{char}}
\nc{\End}{\operatorname{End}}
\nc{\Fil}{\operatorname{Fil}}
\nc{\Mor}{\operatorname{Mor}\xspace}
 \nc{\gmzvs}{gMZV\xspace}
\nc{\gmzv}{gMZV\xspace}
 \nc{\mzv}{MZV\xspace}
\nc{\mzvs}{MZVs\xspace}
\nc{\Hom}{\operatorname{Hom}}
\nc{\id}{\operatorname{id}}
\nc{\im}{\operatorname{im}}
 \nc{\incl}{\operatorname{incl}}
 \nc{\map}{\operatorname{Map}}
\nc{\mchar}{\oparatorname{char}}
\nc{\nz}{\rm NZ}
 \nc{\supp}{\mathrm Supp}
 \nc{\Int}{\mathbf{Int}}
\nc{\Mon}{\mathbf{Mon}}

\nc{\Alg}{\mathbf{Alg}} \nc{\Bax}{\mathbf{Bax}} \nc{\bff}{\mathbf f}
\nc{\bfk}{{\bf k}} \nc{\bfone}{{\bf 1}} \nc{\bfx}{\mathbf x}
\nc{\bfy}{\mathbf y}
\nc{\base}[1]{\bfone^{\otimes ({#1}+1)}} 
\nc{\Cat}{\mathbf{Cat}}

\nc{\detail}{\marginpar{\bf More detail}
    \noindent{\bf Need more detail!}
    \svp}

\nc{\rbtm}{{shuffle }} \nc{\rbto}{{Rota-Baxter }}
\nc{\remarks}{\noindent{\bf Remarks: }} \nc{\Rings}{\mathbf{Rings}}
\nc{\Sets}{\mathbf{Sets}} \nc{\wtot}{\widetilde{\odot}}
\nc{\wast}{\widetilde{\ast}} \nc{\bodot}{\bar{\odot}}
\nc{\bast}{\bar{\ast}} \nc{\hodot}[1]{\odot^{#1}}
\nc{\hast}[1]{\ast^{#1}} \nc{\mal}{\mathcal{O}}
\nc{\tet}{\tilde{\ast}} \nc{\teot}{\tilde{\odot}}
\nc{\oex}{\overline{x}} \nc{\oey}{\overline{y}}
\nc{\oez}{\overline{z}} \nc{\oef}{\overline{f}}
\nc{\oea}{\overline{a}} \nc{\oeb}{\overline{b}}
\nc{\weast}[1]{\widetilde{\ast}^{#1}}
\nc{\weodot}[1]{\widetilde{\odot}^{#1}} \nc{\hstar}[1]{\star^{#1}}
\nc{\lae}{\langle} \nc{\rae}{\rangle}
\nc{\lf}{\lfloor}\nc{\rf}{\rfloor}

\nc{\cala}{{\mathcal A}} \nc{\calb}{{\mathcal B}}
\nc{\calc}{{\mathcal C}}
\nc{\cald}{{\mathcal D}} \nc{\cale}{{\mathcal E}}
\nc{\calf}{{\mathcal F}} \nc{\calg}{{\mathcal G}}
\nc{\calh}{{\mathcal H}} \nc{\cali}{{\mathcal I}}
\nc{\call}{{\mathcal L}} \nc{\calm}{{\mathcal M}}
\nc{\caln}{{\mathcal N}} \nc{\calo}{{\mathcal O}}
\nc{\calp}{{\mathcal P}} \nc{\calr}{{\mathcal R}}
\nc{\cals}{{\mathcal S}} \nc{\calt}{{\mathcal T}}
\nc{\calu}{{\mathcal U}} \nc{\calw}{{\mathcal W}} \nc{\calk}{{\mathcal K}}
\nc{\calx}{{\mathcal X}} \nc{\CA}{\mathcal{A}}

\nc{\fraka}{{\mathfrak a}} \nc{\frakA}{{\mathfrak A}}
\nc{\frakb}{{\mathfrak b}} \nc{\frakB}{{\mathfrak B}}
\nc{\frakD}{{\mathfrak D}} \nc{\frakg}{{\mathfrak g}}
\nc{\frakH}{{\mathfrak H}} \nc{\frakL}{{\mathfrak L}}
\nc{\frakM}{{\mathfrak M}} \nc{\bfrakM}{\overline{\frakM}}
\nc{\frakm}{{\mathfrak m}} \nc{\frakP}{{\mathfrak P}}
\nc{\frakN}{{\mathfrak N}} \nc{\frakp}{{\mathfrak p}}
\nc{\frakS}{{\mathfrak S}}

\font\cyr=wncyr10 \font\cyrs=wncyr7
\nc{\gl}[1]{\textcolor{red}{LG:#1}}
\nc{\ql}[1]{\textcolor{blue}{QL:#1}}
\nc{\red}[1]{\textcolor{red}{#1}}
\nc{\blue}[1]{\textcolor{blue}{#1}}

\allowdisplaybreaks[4]
\makeatletter
\g@addto@macro{\endabstract}{\@setabstract}
\newcommand{\authorfootnotes}{\renewcommand\thefootnote{\@fnsymbol\c@footnote}}%
\makeatother
\begin{document}
\begin{center}
  \LARGE
Modules of polynomial Rota-Baxter Algebras and matrix equations\par \bigskip

  \normalsize
  \authorfootnotes
\normalsize
  \authorfootnotes
 Xiaomin Tang$^{1,2,}$ \footnote{Corresponding author: {\it X. Tang. Email:} tangxm@hlju.edu.cn}
\par \bigskip
   \textsuperscript{1}School of  Mathematical Science, Heilongjiang University,
Harbin, 150080, P. R. China   \par

\textsuperscript{2} School of  Mathematical Science, Harbin Engineering University,
Harbin, 150001, P. R. China   \par

\end{center}

%
\begin{abstract} The all Rota-Baxter algebra structures on the polynomial algebra $R={\bf k}[x]$ are well known.
We study the finite dimensional modules of polynomial Rota-Baxter algebras $(\bfk[x],P)$ or $(x {\bf k} [x],P)$ of weight  nonzero since some cases of weight  zero have been studied. The main result shows that every module over the polynomial Rota-Baxter algebra $(\bfk[x],P)$ or $(x {\bf k} [x],P)$ is equivalent to the modules over a plane ${\bf k}\langle x,y \rangle/ I$ where $I$ is some ideal of free algebra ${\bf k}\langle x,y \rangle$.  Furthermore, we provide the classification of modules of polynomial Rota-Baxter algebras  of weight nonzero through solution to some matrix equation.
\end{abstract}

\noindent{\it Keywords:} Module, Rota-Baxter algebra; matrix equation.
\vspace{2mm}

\noindent{\it AMS subjclass (2000)}: 16W99;  45N05; 12H20



\setcounter{section}{0}

\section{Introduction}

A { Rota-Baxter algebra} (first known as a Baxter algebra) is an
associative algebra $R$ with a linear operator $P$ on $R$ that satisfies the { Rota-Baxter identity}
\begin{equation}
P(r)P(s)=P(P(r)s)+P(rP(s))+\lambda P(rs)\ \ \forall r,s\in R,
\label{eq:rbo}
\end{equation}
where $\lambda$, called the { weight}, is a fixed element in the base ring of the algebra $R$. The operator $P$ is called a Rota-Baxter operator of weight $\lambda$ on $R$.  We usually also say that the pair $(R,P)$ is a Rota-Baxter algebra of weight $\lambda$.

The study of Rota-Baxter algebras (operators) originated in the work of Baxter \cite{Ba} on fluctuation theory, and the algebraic study was started by Rota \cite{Rot1}. The theory of Rota-Baxter algebras develops a general framework of the algebraic and combinatorial structures underlying integral calculus which is like differential algebras for differential calculus. Rota-Baxter algebra also finds its applications in combinatorics, mathematics physics, operads and number
theory \cite{Ag3,A-M,And,BBGN,Ca,E,E-G1,E-G4,EGK3,EGM,G-Z,G-K1,G-K2,tang2019}. See~\cite{Guw,Gub} for further details.

As in the case of common algebraic structures such as associative algebras and Lie algebras, it is important to study the modules and representations of Rota-Baxter algebras. However, the module (or representation) theory of Rota--Baxter algebras is still in the early stage of development. The concept of modules (or representations) of Rota-Baxter algebras was introduced in \cite{G-Lin}. Further studies in this direction were pursued in~\cite{GGQ,LQ,QP} on regular-singular decompositions, geometric representations and derived functors of Rota-Baxter modules, especially those over the Rota-Baxter algebras of Laurent series and polynomials.

The polynomial algebra $\mathbf{k}[x]$ is an important object both in analysis and in algebra. It provides an ideal testing ground to see how an abstractly defined Rota-Baxter operator is related to the integration operator, because of its analytic connection, as functions, and its algebraic significance as a free object in the category of $\mathbf{k}$-algebras. Recently, the authors in \cite{rbopaiaoguo,yu} study the Rota-Baxter operators on the polynomial algebra $\mathbf{k}[x]$ that send monomials to monomials. Due to Theorem 3.5 in \cite{yu}, we have the following result.

\begin{prop}\label{thmnonzero}
Let $P$ be a nonzero  monomial linear operator on $\mathbf{k}[x]$.
Then $P$ is a Rota-Baxter operator of weight $\lambda\neq 0$ if and only if $P$ is one of the following cases:

\begin{enumerate}
\item\label{weightnzero1}  there exists $b\in \mathbf{k}\backslash\{0\}$ such that $P(x^n)=(-\lambda)^{1-n}b^n$ for all $n\in\mathbb{N}$;

\item\label{weightnzero2}   $P(x^n)=-\lambda x^n$ for all $n\in\mathbb{N}$;

\item\label{weightnzero3}   for all $n\in\mathbb{N}$,
\begin{equation*}
P(x^n)=
\begin{cases}
0,& n=0,\\
-\lambda x^n,& n\neq0;
\end{cases}
\end{equation*}

\item\label{weightnzero4}  for all $n\in\mathbb{N}$,
\begin{equation*}
P(x^n)=
\begin{cases}
-\lambda,& n=0,\\
0,& n\neq0.
\end{cases}
\end{equation*}
\end{enumerate}
\end{prop}

On the other hand, recall that the authors in~\cite{QP} study the modules over a class of polynomial Rota-Baxter algebras of weight zero by studying the modules over the Jordan plane.  This inspires us to study the modules of polynomial Rota-Baxter algebras of weight  nonzero.

The main aim of this study is to investigate finite dimensional $(\bfk[x],P)$-modules or $(x\bfk[x],P)$-modules of weight nonzero. As is pointed out in ~\cite{QP} that the category of modules over a differential algebra is equivalent to the category of modules over its corresponding algebra of differential operators. Thus, our first step involves transforming the problem concerning modules of $(\bfk[x],P)$  or $(x\bfk[x],P)$  into the problem of representing some certain types of algebra in the usual sense. These problems are considered in Subsections \ref{Jordan} and \ref{txm-->Ozyy}, where we show that this problem is reduced to the case of $\lambda=1$ and so study the $({\bf k}[x],P)$-modules or $(x\bfk[x],P)$-modules are  equivalent to studying the modules of ${\bf k}\langle x,y \rangle/I$  for some  ideal of free algebra ${\bf k}\langle x,y \rangle$, which are not any special example of Ore extensions of ${\bf k}[x]$ and so the theory on such module (even on structure) has not attracted people's attention. This is different from the case of weight zero \cite{QP}. Finally, by solving some matrix equations we determine  corresponding module structures.

In what follows we assume that ${\bf k}$ is an algebraically closed fields of characteristic zero. Recall that the symbols $\mathbb{Z}$ and $\mathbb{N}$ represent the sets of  integers and nonnegative integers respectively.

\section{Rota-Baxter modules of $({\bf k}[x],P)$ } \label{Jordan}

In this section, after some basic definitions, we show that the modules of $(\bfk [x],P)$ are equivalent to modules of the planes $\mathcal{J}_1$ or $\mathcal{J}_2$.

\subsection{Modules of $(\bfk[x],P)$} \label{Jordan1}
\begin{defn}
{\rm Let $\bfk$ be a field and $(R,P)$ a Rota-Baxter $\bfk$-algebra of weight $\lambda$. A (left) Rota-Baxter module of
$(R,P)$ or simply an (left) $(R,P)$-module is a pair $(M,p)$, where $M$ is an
 $R$-module and $p: M \longrightarrow M$ is a $\bfk$-linear map that satisfies
\begin{equation}\label{eq:meq}
P(r)p(m) = p(P(r)m) +p(rp(m))+\lambda p(rm), \quad \forall ~r\in R,  \forall m\in M.
\end{equation}
}
\end{defn}
If we let $(M,p)$ be an $(R,P)$-module, then $M$ is a $\bfk$-vector space.
If $\dim_{\bfk}M < +\infty$, then $(M,p)$ is called a finite dimensional
$(R,P)$-module. In the following, all $(R,P)$-modules are
assumed to be finite dimensional.

\begin{remark} (see \cite{G-Lin})
This definition of Rota-Baxter module is consistent with the Eilenberg's approach to the definition of module, namely the semidirect sum $(R\oplus M, P+p)$ is a Rota-Baxter algebra. Moreover, its quotient by the Rota-Baxter ideal $(M,p)$ is isomorphic to the initial Rota-Baxter algebra $(R,P)$.
\end{remark}

\begin{prop}\label{tangtang1}
Suppose that $\lambda\neq 0$. Then $(M,p)$ is a module of Rota-Baxter algebra $(R,P)$ of weight $\lambda$ if and only if  $(M,\lambda^{-1}p)$ is a module of Rota-Baxter algebra $(R,\lambda^{-1}P)$ of weight $1$.
\end{prop}

\begin{proof}
If we multiply both sides of Equations (\ref{eq:rbo}) and (\ref{eq:meq}) by $\lambda^{-2}$ respectively, then we obtain
\begin{equation*}
(\frac{1}{\lambda}P)(r)(\frac{1}{\lambda}P)(s)=(\frac{1}{\lambda}P)((\frac{1}{\lambda}P)(r)s)+(\frac{1}{\lambda}P)(r(\frac{1}{\lambda}P)(s))+(\frac{1}{\lambda}P)(rs), \ \forall r,s\in R,
\end{equation*}
and
\begin{equation*}
(\frac{1}{\lambda}P)(r)(\frac{1}{\lambda}p)(m) = (\frac{1}{\lambda}p)((\frac{1}{\lambda}P)(r)m) +(\frac{1}{\lambda}p)(r(\frac{1}{\lambda}p)(m))+ (\frac{1}{\lambda}p)(rm), \ \forall r\in R,  \forall m\in M.
\end{equation*}
This, together with Equations (\ref{eq:rbo}) and (\ref{eq:meq}), yields the conclusion.
\end{proof}

Proposition \ref{tangtang1} tells us that the study of module of Rota-Baxter algebra $(R,P)$ of weight $\lambda$ is reduced to the case of $\lambda=1$.
Since we shall study the Rota-Baxter $(\bfk[x],P)$ modules of weight nonzero, so we will assume without loss of generality that $(\bfk[x],P)$ is a
Rota-Baxter algebra of weight $1$.

\begin{remark}\label{tangtang2}
By Proposition \ref{thmnonzero}, we will mainly consider the four Rota-Baxter algebras $(\bfk[x],P_i)$ of weight $1$,  where Rota-Baxter operators $P_i$, $i=1,2,3,4$ are defined by following:
\begin{enumerate}
\item\label{weightnzero1}  there exists $b\in \mathbf{k}\backslash\{0\}$ such that $P_1(x^n)=(-1)^{1-n}b^n$ for all $n\in\mathbb{N}$;

\item\label{weightnzero2}   $P_2(x^n)=- x^n$ for all $n\in\mathbb{N}$;

\item\label{weightnzero3}   for all $n\in\mathbb{N}$,
\begin{equation*}
P_3(x^n)=
\begin{cases}
0,& n=0,\\
- x^n,& n\neq0;
\end{cases}
\end{equation*}

\item\label{weightnzero4}  for all $n\in\mathbb{N}$,
\begin{equation*}
P_4(x^n)=
\begin{cases}
-1,& n=0,\\
0,& n\neq0.
\end{cases}
\end{equation*}
\end{enumerate}
\end{remark}

Basic concepts regarding an $R$-module can be defined in a similar manner for $(R,P)$-modules.
In particular, an $(R,P)$-module homomorphism
$\phi : (M,p) \longrightarrow (N,q)$ is an $R$-module homomorphism
$\phi: M \longrightarrow M$ that satisfies
$$
\phi \circ p = q \circ \phi.
$$
Furthermore, $(M,p)$ is isomorphic to $(N,q)$ if the homomorphism $\phi$ is bijective.
It is simple to check that $\bigoplus_{i=1}^{n}(M_{i},p_{i}) = (\bigoplus_{i=1}^{n} M_{i},\displaystyle  \sum_{i=1}^{n} p_{i}),$ where $\displaystyle \sum_{i=1}^{n}p_{i}$ is defined by
$$
\left(\sum_{i=1}^{n}p_{i} \right) (u_{1},\cdots ,u_{n}) = \sum_{i=1}^{n}p_{i}(u_{i}),
$$
is still an $(R,P)$-module and it is called the direct sum of $(R,P)$-modules $(M_{1},p_{1}), \cdots, (M_{n},p_{n})$.

\subsection{The planes $\mathcal{J}_1$ and $\mathcal{J}_2$ } \label{Jordan2}
\

Let $(\bfk[x],P_i), i=1,2,3,4$ be the polynomial Rota-Baxter algebras given by Remark \ref{tangtang2} and $\bfk \langle x,y\rangle$ the noncommutative polynomial algebra with variables $x$ and $y$. Let ${\Ixq}_i$ be the ideal of $\bfk \langle x,y\rangle$ generated by the set
\begin{equation}
\mathcal{X}_{i}=\{P_i(f)y-yP_i(f)-yfy-yf~|~f\in \bfk[x]\},
\end{equation}
and $\mathcal{J}_i = \bfk \langle x,y\rangle/{\Ixq}_i,\  i=1,2,3,4$.

For any $m\in \mathbb{N}$, by a simple computation we have
$$
P_i(x^{m})y-yP_i(x^{m})-yx^{m}y-yx^m=-yx^my-yx^{m}
$$
for $i=1,4$ and
$$
P_j(x^{m})y-yP_j(x^{m})-yx^{m}y-yx^m=-yx^my-x^{m}y
$$
for $j=2,3$.  Note that operators $P_i, i=1,2,3,4$ are $\bfk$-linear, so the set
$$
\widetilde{\mathcal{X}_{1}}=\{yx^{m}+yx^my~|~m\in \mathbb{N}\}
$$
also generates ${\Ixq}_i$ for every $i\in \{1,4\}$ and the set
$$
\widetilde{\mathcal{X}_{2}}=\{x^{m}y+yx^my~|~m\in \mathbb{N}\}
$$
also generates ${\Ixq}_j$ for every $j\in \{2,3\}$. Further, we have

\begin{lemma}\label{lem:simplify}
If we let $\widehat{\mathcal{X}_1}=\{y+y^2, xy+yxy\}$ and  $\widehat{\mathcal{X}_2}=\{y+y^2, yx+yxy\}$, then ${\Ixq}_i$ is generated by $\widehat{\mathcal{X}_1}$ for $i=1,4$ and
${\Ixq}_j$ is generated by $\widehat{\mathcal{X}_2}$ for $j=2,3$. Namely, we have
$$ \mathcal{J}_1= \mathcal{J}_4={\bf k}\langle x,y \rangle/(y+y^2, xy+yxy),\ \ \ \mathcal{J}_2= \mathcal{J}_4={\bf k}\langle x,y \rangle/(y+y^2, yx+yxy).$$
\end{lemma}
\begin{proof}
Denote by $\widehat{{\Ixq}_1}$ or $\widehat{{\Ixq}_2}$ the ideal of $\bfk \langle x,y\rangle$ generated by the set  $\widehat{\mathcal{X}_1}$ or $\widehat{\mathcal{X}_2}$ respectively. As above we know that
$
\widetilde{\mathcal{X}_{1}}=\{yx^{m}+yx^my~|~m\in \mathbb{N}\}
$
generates ${\Ixq}_i$ for every $i\in \{1,3\}$ and
$
\widetilde{\mathcal{X}_{2}}=\{x^{m}y+yx^my~|~m\in \mathbb{N}\}
$ generates ${\Ixq}_j$ for every $j\in \{2,4\}$.
Therefore, it is clear by $\widehat{\mathcal{X}_1}\subset \widetilde{\mathcal{X}_{1}}$ and
$\widehat{\mathcal{X}_2}\subset \widetilde{\mathcal{X}_{2}}$ that $\widehat{{\Ixq}_1}\subseteq {\Ixq}_i$ for $i\in \{1,3\}$
and $\widehat{{\Ixq}_2}\subseteq {\Ixq}_j$ for $j\in \{2,4\}$.

 Conversely, we need to show that $\widehat{{\Ixq}_1}\supseteq {\Ixq}_i$ for $i\in \{1,3\}$
and $\widehat{{\Ixq}_2}\supseteq {\Ixq}_j$ for $j\in \{2,4\}$. We use induction on $m$ to show that
$yx^{m}+yx^my\in \widehat{{\Ixq}_1}$ and $x^{m}y+yx^my\in \widehat{{\Ixq}_2}$.
The cases for $m=0,1$ are obvious.  Suppose that any $m \geq 1$, $yx^{m}+yx^my\in \widehat{{\Ixq}_1}$ and
$x^{m}y+yx^my\in \widehat{{\Ixq}_2}$.    Note that
\begin{eqnarray*}
&&yx^{m+1}y+yx^{m+1} \\
&&= (yxy+yx)(x^{m}y+x^{m})-yx(yx^{m}y+yx^{m}),
\end{eqnarray*}
and
\begin{eqnarray*}
&&yx^{m+1}y+x^{m+1}y \\
&&= (yx^{m}+x^{m})(yxy+xy)-(yx^{m}y+x^{m}y)xy.
\end{eqnarray*}
Then, by the induction hypothesis, we obtain $yx^{m+1}y+yx^{m+1} \in \widehat{{\Ixq}_1}$ and $yx^{m+1}y+x^{m+1}y \in \widehat{{\Ixq}_2}$. Therefore, we have $\widehat{{\Ixq}_1}\subseteq {\Ixq}_i$ for $i\in \{1,4\}$
and $\widehat{{\Ixq}_2}\subseteq {\Ixq}_j$ for $j\in \{2,3\}$. The proof is completed.
\end{proof}

\begin{coro}\label{lem:ker}
For any $m,k \in \mathbb{N}$, the equation $yx^{m} = y x^{m}(-y)^k$ holds in $\mathcal{J}_i$ and the equation $x^{m}y = (-y)^kx^{m}y$ holds in $\mathcal{J}_j$ for $i\in \{1,4\}$ and $j\in \{2,3\}$.
\end{coro}
\begin{proof}
By the above results, one has for $m \in \mathbb{N}$, the equation $yx^{m} = y x^{m}(-y)$ holds in $\mathcal{J}_i, i=1,4$ and $x^{m}y = (-y)x^{m}y$  holds in $\mathcal{J}_j, j=2,3$.
It  follows that conclusion.
\end{proof}

\subsection{The relations between $\mathcal{J}_i$-modules and $(\bfk[x],P_i)$-modules}
\

Now we establish the relationship between $\mathcal{J}_i$-modules and $(\bfk[x],P_i)$-modules for $i=1,2,3,4$. Recall that a $(\bfk[x],P_i)$-module is a pair $(M,p_i)$, where $M$ is a $\bfk[x]$-module and $p_i \in {\rm End}_{\bfk}(M)$ such that
\begin{equation}
P_i(f)p_i(v)=p_i(P_i(f)v+fp_i(v)+f v), \quad \forall f \in \bfk[x],~~ \forall v \in M.
\end{equation}

\begin{prop}\label{prop:modulerel}
Take $i\in \{1,2,3,4\}$. Let $M$ be a $\mathcal{J}_i$-module. Define a $\bfk$-linear map $p_i$ on $M$ by
$$
p_i(v) = yv, \quad \forall v \in M.
$$
Then, $(M,p_i)$ is a $(\bfk[x],P_i)$-module. Conversely, if $(M,p_i)$ is a $(\bfk[x],P_i)$-module and we define
$$
yv = p_i(v), \quad \forall v \in M,
$$
then $M$ is a $\mathcal{J}_i$-module.
\end{prop}
\begin{proof}
We note that the equations
$$
P_i(x^{m})y-yP_i(x^{m})-yx^{m}y-yx^m =0, \ m \in \mathbb{N}
$$
hold in $\mathcal{J}_i$. Thus, for any $v \in M$, we have
$$
(P_i(x^{m})y-yP_i(x^{m})-yx^{m}y-y x^m)(v)=0,
$$
i.e.,
$$
P_i(x^{m})p_i(v)=p_i(P_i(x^{m})v+x^{m}p_i(v)+x^m v).
$$
Hence, $(M,p_i)$ is a $(\bfk[x],P_i)$-module.

Conversely, suppose that $M$ is a $\bfk\langle x,y \rangle$-module. Since $(M,p_i)$ is a $(\bfk[x],P_i)$-module, so we have
$$
{\Ixq}_i \subseteq {\rm ann} M = \{F \in \bfk\langle x,y\rangle~|~Fv=0, \mbox{for all}~v \in M\}.
$$
Thus, $M$ is a $\mathcal{J}_i$-module.
\end{proof}

Due to Proposition \ref{prop:modulerel}, the study of $(\bfk[x],P_i)$-modules becomes the study of $\mathcal{J}_i$-modules in the usual sense for $i\in \{1,2,3,4\}$. By Proposition \ref{prop:modulerel} and Lemma \ref{lem:simplify}, we have the following

\begin{coro}\label{thm:sim}
Let $M$ be a $\bfk[x]$-module and $p_i \in {\rm End}_{\bfk}(M)$. Then for $i\in \{1,2,3,4\}$, $(M,p_i)$ is a $(\bfk[x],P_i)$-module if and only if $p_i^2=-p_i$ and
\begin{equation}\label{equ:main}
p_ix=-p_ixp_i \text{\ if } i=1,4; \text{ or } xp_i=-p_ixp_i \text{\ if } i=2,3.
\end{equation}
\end{coro}

\begin{coro}\label{coro:onedim}
If $(M,p_i)$ is a $1$-dimensional $(\bfk[x],P_i)$-module, then $p_i=0$ or $p_i=-I_M,\ i=1,2,3,4$.
\end{coro}

\subsection{The classicization of $(\bfk[x],P_i)$-modules} \label{class}
\

Take $i\in \{1,2,3,4\}$. Proposition \ref{prop:modulerel} shows that studying $(\bfk[x],P_i)$-modules is equivalent to studying modules of the plane $\mathcal{J}_i$-module in the usual sense. Thus, descriptions of the $\mathcal{J}_i$-module can be interpreted in terms of the $(\bfk[x],P)$-module.
Any $\mathcal{J}_i$-module can be regarded as both a $\bfk[x]$-module and a $\bfk[y]$-module, but the role of action $x$ is different from the role of action $y$. Our method aims to determine the $\mathcal{J}_i$-module structures on a given $\bfk[y]$-module with the action $y$.

For a $\bfk[x]$-module $M$, let $\tau(v)=xv, \forall v \in M$, and thus $\tau \in {\rm End}_{\bfk}(M)$. It is well known that a $\bfk[x]$-module $M$ can be regarded as a $\bfk$-vector space $M$ with a $\bfk$-linear map $\tau \in {\rm End}_{\bfk}(M)$ and $f(x)v=f(\tau)v$, $f(x) \in \bfk[x]$. In the following, the linear map induced by the action of $x$ is always denoted by $\tau$. Therefore, a $(\bfk[x],P)$-module $(M,p_i)$ can be regarded as a $\bfk$-vector space $M$ with two $\bfk$-linear maps, $\tau$ and $p_i$. By Corollary \ref{thm:sim}, we obtain the following:

\begin{prop}\label{tangtang3}
Take $i\in \{1,2,3,4\}$. Let $M$ be a $\bfk[x]$-module, $p_i:M\longrightarrow M$ a $\bfk$-linear map, and fix a $\bfk$-basis $v_{1}, v_{2}, \cdots, v_{n}$ of $M$. Let the matrices of $\tau$ and $p_i$ corresponding to the basis $v_{1},v_{2},\cdots,v_{n}$ be $A$ and $B$, respectively. Then, $(M,p_i)$ is a $(\bfk[x],P_i)$-module if and only if
$B^2=-B$ and
$$
BA=-BAB \text{\ if } i=1,4; \text{ or } AB=-BAB \text{\ if } i=2,3.
$$
\end{prop}

\begin{remark}
Recall that a $\bfk$-linear operator $p$ on a module $M$ is called {\it quasi-idempotent of weight $0\neq \lambda\in \bfk$} if $p^2+\lambda p=0$. For $\mu\in \bfk$, let
\begin{equation*}
M_\mu\colon = \{x\in M\,|\, p(x)=\mu x\}
\end{equation*}
denote the eigenspace of $M$ for the eigenvalue $\mu$.
A Rota-Baxter operator $P$ of weight $\lambda$ in a $\bfk$-algebra $R$  is called {\it quasi-idempotent}~\cite{A-M} if $P^2+\lambda P=0$. If $p$ is quasi-idempotent, then we have the decomposition $M=M_{-\lambda} \oplus M_0$ which will be called the {\it regular-singular decomposition}. A more detailed study in this case can be found in~\cite{LQ}.  From Corollary  \ref{thm:sim} we know that $p_i$  in $(\bfk[x],P_i)$-module $(M,p_i)$ is quasi-idempotent with $\lambda=1$.
\end{remark}

\begin{theorem} \label{tang1}
Suppose that $i\in \{1,2,3,4\}$ and $n$ is a positive integer. Then $(M,p_i)$ is a $n$ dimensional $(\bfk[x],P_i)$-module if and only if the matrices $A$ and $B$ of $\tau$ and $p_i$ corresponding to a $\bfk$-basis of $M$ have the following forms:

(i) When $i=1, 4$, for some $k\in \mathbb{N}$,
$$A=\begin{bmatrix} A_1 & 0\\ A_3& A_4 \end{bmatrix}, \ \ \ B=\begin{bmatrix} -I_k & \\ & 0 \end{bmatrix} $$
where $A_1\in M_k(\bfk)$ and $A_4\in M_{n-k}(\bfk)$;

(ii) When $i=2, 3$, for some $k\in \mathbb{N}$,
$$A=\begin{bmatrix} A_1 & A_2\\ 0& A_4 \end{bmatrix}, \ \ \ B=\begin{bmatrix} -I_k & \\ & 0 \end{bmatrix} $$
where $A_1\in M_k(\bfk)$ and $A_4\in M_{n-k}(\bfk)$.
\end{theorem}

\begin{proof}
If $(M,p_i)$ is a $n$ dimensional $(\bfk[x],P_i)$-module, by Proposition \ref{tangtang3} we know that $B^2=-B$.
Then $B$ is similar to a diagonal matrix with diagonal elements $-1$ or $0$. Select the appropriate $\bfk$-basis of $M$, we can assume that $B$ is of the form
$$
B=\begin{bmatrix} -I_k & \\ & 0 \end{bmatrix},
$$
where $k$ is the rank of $B$; and here we specify that $I_0=0$. Now let $A$ is of the form
$$A=\begin{bmatrix} A_1 & A_2\\ A_3& A_4 \end{bmatrix}$$
where $A_1\in M_k(\bfk)$ and $A_4\in M_{n-k}(\bfk)$.
When $i=1, 4$, by Proposition \ref{tangtang3} we also have $BA=-BAB$. From this we obtain that $A_2=0$. When $i=2, 3$, by Proposition \ref{tangtang3} we have $AB=-BAB$ and so one has $A_3=0$.  This proves the necessity. The proof of sufficiency can be verified directly. \end{proof}

Theorem \ref{tang1} gives the complete classification of Rota-Baxter modules of polynomial Rota-Baxter of weigh nonzero (translated into the case of weigh $1$ ).  This problem is completed by resolving matrix equations $B^2=-B$ with $BA=-BAB$ or $AB=-BAB$. It is relatively simple since the equation $B^2=-B$ makes the form of $B$ is easy to determine. Next section we will consider the modules of Rota-Baxter algegra $(x {\bf k} [x], P)$ and we will solve matrix equation including only $AB=-BAB$ which makes the problem to be interesting.  It should be pointed out that the irreducible modules of Rota-Baxter algegra $( {\bf k} [x], P)$ also can be discussed in a similar way to  $( x{\bf k} [x], P)$-modules, see Corollary \ref{tangxzhong}.

\section{Modules of Rota-Baxter algebra $(x {\bf k} [x], P)$} \label{txm-->Ozyy}

The polynomial algebra $\bfk [x]$ has an important subalgebra as
$$x {\bf k} [x]=\{xf| f\in \bfk [x]  \}\cong \bfk[x] / \bfk.$$
In this section, we will study the modules of the subalgebra $(x {\bf k} [x], P)$ of Rota-Baxter algegra $({\bf k} [x], P)$ of weigh nonzero.  Here we take the Rota-Baxter operator $P$ on $x {\bf k} [x]$ by  restriction of Rota-Baxter operator of ${\bf k} [x]$ given by Proposition \ref{thmnonzero}. In view of Proposition \ref{tangtang1}, below we consider the Rota-Baxter algegra $(x{\bf k} [x], P)$ with the Rota-Baxter operator $P: x{\bf k} [x]\rightarrow x{\bf k} [x]$ of weight $1$ given by $$P(x^n)=-x^n$$ for all $n\in \mathbb{Z}$ with $n\ge 1$.

\subsection{The describe of modules of Rota-Baxter algebra $(x {\bf k} [x], P)$}
\

Let ${\Ixq}$ be the ideal of $\bfk \langle x,y\rangle$ generated by the set
\begin{equation}
\mathcal{X}=\{P(f)y-yP(f)-yfy-yf~|~f\in x\bfk[x]\},
\end{equation}
and $\mathcal{J} = \bfk \langle x,y\rangle/{\Ixq}$.

For any $m\in \mathbb{Z}$ with $m\ge 1$,  we have
$$
P(x^{m})y-yP(x^{m})-yx^{m}y-yx^m=-yx^my-x^{m}y.
$$
Note that operators $P$ is $\bfk$-linear, so the set
$$
\widetilde{\mathcal{X}}=\{x^{m}y+yx^my~|~m\ge 1\}
$$
also generates ${\Ixq}$. Further, similar to the proof of Lemma \ref{lem:simplify} we have:

\begin{lemma}\label{lem:simplify1}
If we let $\widehat{\mathcal{X}}=\{xy+yxy\}$, then ${\Ixq}$ is generated by $\widehat{\mathcal{X}}$. Namely, we have
$$ \mathcal{J}={\bf k}\langle x,y \rangle/(xy+yxy).$$
\end{lemma}

Now we establish the relationship between $\mathcal{J}$-modules and $(x\bfk[x],P)$-modules. Similar to Proposition \ref{prop:modulerel}, we get
\begin{prop}\label{prop:modulerel1}
Let $M$ be a $\mathcal{J}$-module. Define a $\bfk$-linear map $p$ on $M$ by
$$
p(v) = yv, \quad \forall v \in M.
$$
Then, $(M,p)$ is a $(x\bfk[x],P)$-module. Conversely, if $(M,p)$ is a $(x\bfk[x],P)$-module and we define
$$
yv = p(v), \quad \forall v \in M,
$$
then $M$ is a $\mathcal{J}$-module.
\end{prop}

Due to Proposition \ref{prop:modulerel1}, the study of $(x\bfk[x],P)$-modules becomes the study of $\mathcal{J}$-modules in the usual sense. By Proposition \ref{prop:modulerel1} and Lemma \ref{lem:simplify1}, for a $x\bfk[x]$-module $M$, let $\tau(v)=xv, \forall v \in M$, and thus $\tau \in {\rm End}_{\bfk}(M)$. Like in Section \ref{Jordan}, the linear map induced by the action of $x$ is always denoted by $\tau$. Therefore, a $(x\bfk[x],P)$-module $(M,p)$ can be regarded as a $\bfk$-vector space $M$ with two $\bfk$-linear maps, $\tau$ and $p$. Similar to Proposition \ref{tangtang3}, we have
\begin{prop}\label{tangtang31}
Let $M$ be a $x\bfk[x]$-module, $p:M\longrightarrow M$ a $\bfk$-linear map, and fix a $\bfk$-basis $v_{1}, v_{2}, \cdots, v_{n}$ of $M$. Let the matrices of $\tau$ and $p$ corresponding to the basis $v_{1},v_{2},\cdots,v_{n}$ be $A$ and $B$, respectively. Then, $(M,p)$ is a $(x\bfk[x],P)$-module if and only if
\begin{equation}\label{zhong1}
AB=-BAB.
\end{equation}
\end{prop}

\subsection{The solution to matrix equation $AB=-BAB$.}

\

In order to give the classification of $(x\bfk[x],P)$-module, we should to solve the matrix equation (\ref{zhong1}). We first give some conclusions.
Recall that the Jordan block of size $k$ is a $k\times k$ matrix of the following form
$$
J_{k}(b) = \begin{bmatrix}
b&1&\cdots&0&0 \\
0&b&\ddots&0&0 \\
\vdots&\vdots&\ddots&\ddots&\vdots \\
0&0&\cdots&b&1 \\
0&0&\cdots&0&b
\end{bmatrix}.
$$
It is clear that $J_{k}(b)=bI_k+J_{k}(0)$ where $J_{k}(0)$ is a nilpotent Jordan block.

\begin{prop}\label{zhongt1}
Suppose that $t,s\in \mathbb{Z}$ with $t,s\ge 1$ and $b_1,b_2\in {\bf k}$. Then $X$ is the solution of the matrix equation
\begin{equation}\label{zyy1}
XJ_t(b_2)=-J_s(b_1) X J_t(b_2)
\end{equation}
if and only if one of the following cases holds:
\begin{enumerate}
\item When $(b_1, b_2)=(-1, 0)$, then $X$ has the form
\begin{equation}\label{tintoz1}
X=\begin{pmat}[{..|.}]
\ast &\cdots& \ast & \ast\cr\-
0&\ldots& 0 & \ast\cr
\vdots&\ddots&\vdots & \vdots \cr
0&\ldots & 0&\ast \cr
\end{pmat}_{s\times t},
\end{equation}
here and below the symbol $\ast$ means it can take any element in ${\bf k}$;
\item  When $b_1=-1$ and $b_2\neq 0$, then $X$ has the form
\begin{equation}\label{tintoz2}
X=\begin{pmat}[{..}]
\ast &\cdots& \ast & \ast\cr\-
0&\ldots& 0 & 0\cr
\vdots&\ddots&\vdots & \vdots\cr
0&\ldots & 0&0\cr
\end{pmat}_{s\times t};
\end{equation}

\item When $b_1\neq -1$ and $b_2=0$, then $X$ has the form
\begin{equation}\label{tintoz3}
X=\begin{pmat}[{..|.}]
0 &\cdots& 0 & \ast\cr
0&\ldots& 0 & \ast\cr
\vdots&\ddots&\vdots & \vdots \cr
0&\ldots & 0&\ast\cr
\end{pmat}_{s\times t};
\end{equation}

\item When $b_1\neq -1$ and $b_2\neq 0$, then $X=0$.

\end{enumerate}
\end{prop}

\begin{proof}
  The proof of sufficiency can be verified directly. Now we prove the necessity.
  Denote $X=[x_{ij}]\in \bfk^{s\times t}$.
  It will be divided into the following $2$ cases according to the values of $b_1,b_2$.

{\it Case 1.}  When $b_1=-1$.  By (\ref{zyy1}), one has $XJ_t(b_2)=-(-I_s+J_s(0)) X J_t(b_2)$ which implies
$$ J_s(0) X J_t(b_2)=0.  $$

If $b_2=0$, then above equation becomes $ J_s(0) X J_t(0)=0, $ that is
$$
\begin{bmatrix}
0&1&\cdots&0 \\
\vdots&\vdots&\ddots&\vdots \\
0&0&\cdots&1 \\
0&0&\cdots&0
\end{bmatrix}_{s\times s} \cdot
\begin{bmatrix}
x_{11}&x_{12}&\cdots&x_{1t} \\
x_{21}&x_{22}&\cdots&x_{2t} \\
\vdots&\vdots&\ddots&\vdots \\
x_{s1}&x_{s2}&\cdots&x_{st}
\end{bmatrix} \cdot
\begin{bmatrix}
0&1&\cdots&0 \\
\vdots&\vdots&\ddots&\vdots \\
0&0&\cdots&1 \\
0&0&\cdots&0
\end{bmatrix}_{t\times t}=0.
$$
By this with a direct computation, we see that $X$ must take the form of (\ref{tintoz1}).

If $b_2\neq 0$, then
the above equation yields $J_s(0) X (b_2I_t+J_t(0))=0$ and so that $J_s(0) X =J_s(0) X(- b_2^{-1}J_t(0))$. Thus, in view of $J_t(0)$ is nilpotent we get
$$
J_s(0) X =J_s(0) X(- b_2^{-1}J_t(0))^2=J_s(0) X(- b_2^{-1}J_t(0))^3=\cdots= J_s(0) X 0=0.
$$
In the other words,
$$
\begin{bmatrix}
0&1&\cdots&0 \\
\vdots&\vdots&\ddots&\vdots \\
0&0&\cdots&1 \\
0&0&\cdots&0
\end{bmatrix}_{s\times s} \cdot
\begin{bmatrix}
x_{11}&x_{12}&\cdots&x_{1t} \\
x_{21}&x_{22}&\cdots&x_{2t} \\
\vdots&\vdots&\ddots&\vdots \\
x_{s1}&x_{s2}&\cdots&x_{st}
\end{bmatrix}
=0.
$$
By the above equation through simple calculation, we obtain that  $X$ must take the form of (\ref{tintoz2}).

{\it Case 2.}  When $b_1\neq -1$. We first claim that the following equation holds:
\begin{equation}\label{tangontozhong}
XJ_t(b_2)=0.
\end{equation}

If $b_1=0$, then Equation (\ref{zyy1}) tells us that $$XJ_t(b_2)=-J_s(0) X J_t(b_2)=(-J_s(0))^2 X J_t(b_2)=\cdots.$$  Since $-J_s(0)$ is nilpotent, we have
Equation (\ref{tangontozhong}) holds.

If $b_1\neq 0$, that is $b_1\neq 0, -1$.  In view of  (\ref{zyy1}), one has
$$
(I_s-b_1^{-1}J_s(b_1))X J_t(b_2)=X J_t(b_2)-b_1^{-1}J_s(b_1)X J_t(b_2)=\frac{b_1+1}{b_1}X J_t(b_2).
$$
Furthermore, by the above equation with $J_s(b_1)=b_1I_s+J_s(0)$, we get
$$
X J_t(b_2)=(-(1+b_1)^{-1}J_s(0))X J_t(b_2)=(-(1+b_1)^{-1}J_s(0))^2X J_t(b_2)=\cdots.
$$
Again since $-(1+b_1)^{-1}J_s(0)$ is nilpotent, we see that Equation (\ref{tangontozhong}) still holds.  The claim is proved.

When $b_2=0$,  by  (\ref{tangontozhong}) we have $XJ_t(0)=0$, i.e.,
$$
\begin{bmatrix}
x_{11}&x_{12}&\cdots&x_{1t} \\
x_{21}&x_{22}&\cdots&x_{2t} \\
\vdots&\vdots&\ddots&\vdots \\
x_{s1}&x_{s2}&\cdots&x_{st}
\end{bmatrix}\cdot
\begin{bmatrix}
0&1&\cdots&0 \\
\vdots&\vdots&\ddots&\vdots \\
0&0&\cdots&1 \\
0&0&\cdots&0
\end{bmatrix}_{t\times t}=0,
$$
which deduces that $X$ must take the form of (\ref{tintoz3}).

When $b_2\neq 0$, by (\ref{tangontozhong}) one has $X(b_2I_t+J_t(0))=0$ and so that
$$
X=X(-b_2^{-1}J_t(0))=X(-b_2^{-1}J_t(0))^2=\cdots=X0=0
$$
since $-b_2^{-1}J_t(0)$ is nilpotent.  The proof is completed.
\end{proof}

Apply Proposition \ref{zhongt1} to $s=t$ and $b_1=b_2$, it follows that
\begin{coro}\label{ttzz}
Suppose that $t\in \mathbb{Z}$ with $t\ge 1$ and $b\in {\bf k}$. Then $X$ is the solution of the matrix equation
\begin{equation}\label{zyy2}
XJ_t(b)=-J_t(b) X J_t(b)
\end{equation}
if and only if one of the following cases holds:
\begin{enumerate}
\item  When $b=-1$, then $X$ has the form (\ref{tintoz2});

\item When $b=0$, then $X$ has the form (\ref{tintoz3});

\item When $b\neq -1, 0$, then $X=0$.
\end{enumerate}
\end{coro}

\begin{theorem} \label{tangXzhongO}
Suppose that $A, B\in M_n({\bf k})$ satisfying $AB=-BAB$. Then there is an invertible matrix $P$ such that
$$
A=P^{-1}
\begin{bmatrix}
X_{11}&X_{12}& \ldots &X_{1l}\\
X_{21}& X_{21}&\ldots&X_{2l}\\
\vdots&\vdots&\ddots&\vdots\\
X_{l1}&X_{l2}&\ldots&X_{l2}
\end{bmatrix}
P
$$
and
$$
B=P^{-1}
\begin{bmatrix}
J_{p_1}(b_1)&0& \ldots &0\\
0& J_{p_2}(b_2)&\ldots&0\\
\vdots&0&\ddots&0\\
0&0&\ldots&J_{p_l}(b_l)
\end{bmatrix}
P
$$
where $b_i\in {\bf k}$, $p_i\in \mathbb{N}\setminus \{0\}$ and $X_{ij}\in {\bf k}^{p_i\times p_j}$  such that
$$
X_{ij}J_{p_j}(b_j)=-J_{p_i}(b_i)X_{ij}J_{p_j}(b_j)
$$
for all $i,j=1, \cdots, l$. Thus $X_{ij}$ can be determined by Proposition \ref{zhongt1}.
\end{theorem}

\begin{proof}  Suppose that $B$ has the Jordan decomposition $B=P^{-1}J P$ where $P$ is an invertible matrix and $J$ is the Jordan canonical form of $B$.
It follows by the condition $AB=-BAB$.
\end{proof}

\subsection{Examples for $(x\bfk [x], P)$-modules.}

\

By Propositions \ref{tangtang31}, \ref{zhongt1} and Theorem \ref{tangXzhongO}, we can give a complete characterization of $(x\bfk [x], P)$-modules. From the view point, some examples are given below.

\begin{exam}
Suppose that $M$ is a $n$-dimensional $\bfk$-vector space and $x, p:M\rightarrow M$ are linear maps with matrices $A$ and $B$ corresponding to an appropriate basis of $M$ respectively.
Then for $n\in \{1,2\}$, the all $n$-dimensional $(x\bfk[x],P)$-module  $(M,p)$ are listed by the following:

\begin{enumerate}
\item  When $n= 1$, (i) $\forall A\in \bfk$, $B=0$; \  (ii) $\forall A\in \bfk$, $B=-1$;\ (iii) $A=0$, $\forall B\in \bfk$;

\item When $n=2$, for any $(a_1,a_2,a_3,a_4,b_2)\in \bfk^4$ with $b_2\neq -1, 0$,

(i) $\forall A\in M_2(\bfk)$, $B=0$; \

(ii) $\forall A\in M_2(\bfk)$, $B=-I_2$;\

(iii) $A=0$, $\forall B\in M_2(\bfk)$;

(iv) $A=\begin{bmatrix} a_1 & 0 \\ a_3 & 0 \end{bmatrix}$,  $B=\begin{bmatrix} 0 & 0 \\ 0 & b_2 \end{bmatrix}$;

 (v) $A=\begin{bmatrix} a_1 & a_2 \\ 0 &a_4  \end{bmatrix}$,  $B=\begin{bmatrix} -1 & 0 \\ 0 & 0 \end{bmatrix}$;

 (vi $A=\begin{bmatrix} 0 & a_2 \\ 0 &a_4  \end{bmatrix}$,  $B=\begin{bmatrix} -1 & 0 \\ 0 & b_2 \end{bmatrix}$;

 (vii) $A=\begin{bmatrix} 0 & a_2 \\ 0 &a_4  \end{bmatrix}$,  $B=\begin{bmatrix} 0& 1 \\ 0 & 0 \end{bmatrix}$;

 (viii) $A=\begin{bmatrix} a_1 & a_2 \\ 0 &0  \end{bmatrix}$,  $B=\begin{bmatrix} -1 & 1 \\ 0 & -1 \end{bmatrix}$.

\end{enumerate}

\end{exam}

\begin{exam}
Suppose that $M$ is a $3$-dimensional $\bfk$-vector space and $x, p:M\rightarrow M$ are linear maps with matrices $A$ and $B$ corresponding to an appropriate basis of $M$ respectively.
If $(M,p)$ is a $3$-dimensional $(x\bfk[x],P)$-module, then we have the following conclusions (where $(x_{ij})\in \bfk^{3\times 3}$):

\begin{enumerate}
\item If $B=\begin{bmatrix} -1& 0&0\\0&-1&0 \\ 0 & 0 &b_3\end{bmatrix}$ where $b_3\in \bfk\setminus\{-1,0\}$, then
$A=\begin{bmatrix} x_{11}& x_{12}&x_{13}\\x_{21}& x_{22}&x_{23} \\ 0 & 0 &0\end{bmatrix}$;

\item If $B=\begin{pmat}[{.|.}]
-1& 1&0\cr
0&-1&0 \cr \-
0 & 0 &b_3 \cr
\end{pmat}$
or $B=\begin{bmatrix} -1& 0&0\\0&b_2&0 \\ 0 & 0 &b_3\end{bmatrix}$ where $b_2,b_3\in \bfk\setminus\{-1,0\}$, or $B=\begin{bmatrix} -1& 1&0\\0&-1&1 \\ 0 & 0 &-1\end{bmatrix}$, then
$A=\begin{bmatrix} x_{11}& x_{12}&x_{13}\\0& 0&0 \\ 0 & 0 &0\end{bmatrix}$;

\item If $B=\begin{bmatrix} 0& 0&0\\0&b_2&0 \\ 0 & 0 &b_3\end{bmatrix}$ where $b_2,b_3\in \bfk\setminus\{-1,0\}$, then
$A=\begin{bmatrix} x_{11}&0&0\\ x_{21}& 0&0 \\  x_{31} & 0 &0\end{bmatrix}$;

\item If $B=\begin{pmat} [{.|.}]
0& 1&0\cr
0&0&0 \cr\-
0 & 0 &b_3\cr
\end{pmat}$
where $b_3\in \bfk\setminus\{-1,0\}$, then
$A=\begin{bmatrix} 0& x_{12}&0\\0& x_{22}&0 \\ 0 & x_{32} &0\end{bmatrix}$;

\item If $B=\begin{bmatrix}b_1& \alpha&0\\0&b_2&\beta \\ 0 & 0 &b_3\end{bmatrix}$  where $b_1,b_2, b_3\in \bfk\setminus\{-1,0\}$ and $\alpha,\beta\in \{0,1\}$, then
$A=0$;

\item If $B=
\begin{pmat}[{.|.}]
-1& 1&0\cr
0&-1&0 \cr\-
0 & 0 &0\cr
\end{pmat}$
 or
 $B=\begin{pmat}[{|..}]
 -1& 0&0\cr\-
 0&0&1\cr
 0 & 0 &0 \cr
 \end{pmat}$,
  or $B=\begin{bmatrix}-1& 0&0\\0&b_2&0 \\ 0 & 0 &0\end{bmatrix}$  where $b_2\in \bfk\setminus\{-1,0\}$, then
$A=\begin{bmatrix} x_{11}& x_{12}&x_{13}\\0& 0&x_{23} \\ 0 & 0 &x_{33}\end{bmatrix}$;

\item If
$B=
\begin{pmat} [{.|.}]
b_1& -1&0 \cr
0&b_1&0 \cr\-
0 & 0 &0 \cr
\end{pmat}$
 where $b_1\in \bfk\setminus\{-1,0\}$, or $B=\begin{bmatrix}0&1&0\\0&0&1\\ 0 & 0 &0\end{bmatrix}$, then
$A=\begin{bmatrix} 0& 0&x_{13}\\0& 0&x_{23} \\ 0 & 0 &x_{33}\end{bmatrix}$.
\end{enumerate}

\end{exam}

\begin{proof}
We without the generality suppose that $B$ is the Jordan canonical form, then $B$ has one, two or three Jordan blocks. When $B$ takes the one form of the cases listed above, then by Propositions \ref{tangtang31}, \ref{zhongt1} and Theorem \ref{tangXzhongO} one can obtain the form of $A$.
\end{proof}

Note that the above class of modules $(M,p)$ determined by matrix pairs of $A, B$ does not contain all of $3$-dimensional $(x\bfk[x],P)$-modules, the few remaining cases can be obtained in the similar way, we omit it here. For $n\ge 4$, we give some examples as follows.

\begin{exam}
Suppose that $M$ is a $n$-dimensional $\bfk$-vector space and $x, p:M\rightarrow M$ are linear maps with matrices $A$ and $B$ corresponding to an appropriate basis of $M$ respectively.
If $(M,p)$ is a $n$-dimensional $(x\bfk[x],P)$-module, then we have the following conclusions (where $(x_{ij})\in \bfk^{n\times n}$):

\begin{enumerate}
\item If $n=4$ and
$B=\begin{pmat}[{.|.}]
 0& 1&0&0\cr
 0&0&0&0 \cr\-
 0 & 0 &-1&1\cr
 0& 0 & 0&-1 \cr
 \end{pmat}$, then
$A=\begin{pmat}[{.|.}]
 0& x_{12}&0&0\cr
 0& x_{22}&0&0 \cr\-
  x_{31} & x_{32} &x_{33}&x_{34}\cr
  0& x_{42}&0&0 \cr
   \end{pmat}$;

\item If $n=5$ and
 $B=\begin{pmat}[{.||.}]
-1& 1&0&0&0\cr
0&-1&0&0 &0\cr\-
 0 & 0 &0&0&0\cr\-
 0& 0 & 0&2&1 \cr
 0&0&0&0&2 \cr
 \end{pmat}$, then
$A=\begin{pmat}[{.||.}]
 x_{11}& x_{12}&x_{13}&x_{14}&x_{15}\cr
   0&0&x_{23}&0&0\cr\-
    0&0&x_{33}&0&0\cr\-
    0& 0&x_{43}&0&0 \cr
     0&0&x_{53}&0&0 \cr
     \end{pmat}$;

\item If $n=6$ and
$B=\begin{pmat}[{.|.|.}]
0& 1&0&0&0&0\cr
0& 0&0&0&0&0\cr\-
 0 & 0 &3&1&0&0\cr
  0& 0 & 0&3&0&0 \cr\-
   0&0&0&0&-1&1\cr
   0&0&0&0&0&-1 \cr
   \end{pmat}$, then
$A=\begin{pmat}[{.|.|.}]
 0& x_{12}&0&0&0&0\cr
  0& x_{22}&0&0&0&0\cr\-
  0& x_{32}&0&0&0&0\cr
  0& x_{42}&0&0&0&0 \cr\-
   x_{51}& x_{52}& x_{53}&x_{54}&x_{55}&x_{56}\cr
    0& x_{62}&0&0&0&0 \cr
   \end{pmat}$.

 \item If $n=2k$ for $k\ge 1$ and
 $B={\rm diag}\left(\begin{bmatrix} -1&1 \\ 0&-1\end{bmatrix}, \cdots, \begin{bmatrix} -1&1 \\ 0&-1\end{bmatrix}\right)$, then
 $
A=
\begin{bmatrix}
X_{11}& \ldots &X_{1k}\\
\vdots&\ddots&\vdots\\
X_{k1}&\ldots&X_{k2}
\end{bmatrix}
$
where $X_{ij}$ is of the form $X_{ij}=\begin{bmatrix}\ast &\ast \\ 0&0 \end{bmatrix}, i, j=1, \cdots, k$, here and below the symbol $\ast$ means it can take any element in ${\bf k}$;

 \item If $n=2k$ for $k\ge 1$ and
 $B={\rm diag}\left(\begin{bmatrix} 0&1 \\ 0&0\end{bmatrix}, \cdots, \begin{bmatrix} 0&1 \\ 0&0\end{bmatrix}\right)$, then
 $
A=
\begin{bmatrix}
X_{11}& \ldots &X_{1k}\\
\vdots&\ddots&\vdots\\
X_{k1}&\ldots&X_{k2}
\end{bmatrix}
$
where $X_{ij}$ is of the form $X_{ij}=\begin{bmatrix}0&\ast\\0&\ast \end{bmatrix}, i, j=1, \cdots, k$.
 \end{enumerate}
\end{exam}

\subsection{Irreducible or indecomposable $(x\bfk [x], P)$-modules. }
\

Now, Proposition \ref{prop:modulerel1} shows that studying $(x\bfk[x],P)$-modules is equivalent to studying modules of the  plane $\mathcal{J}$ in the usual sense.
Thus, descriptions of the irreducible $\mathcal{J}$-module can be interpreted in terms of the $(x\bfk[x],P)$-module.  In this section we will study the irreducible and indecomposable $(x\bfk [x], P)$-modules and the same results also are available for $\mathcal{J}$-module. We will prove that there exists only $1$-dimensional irreducible $(x\bfk [x], P)$-module (or $\mathcal{J}$-module), but the indecomposable $(x\bfk [x], P)$-module can be of any dimension.

\begin{defn} Let $(R,P)$ is a Rota-Baxter algebra.
A nonzero $(R,P)$-module $(M,p)$ is called irreducible if the submodule of $(M,p)$ is either $(0,p)$ or $(M,p)$. $(M,p)$ is called indecomposable if $M \neq 0$ and $(M,p)$ is not the direct sum of its two proper submodules.
\end{defn}

\begin{theorem}\label{tcaoz}
Let $(M,p)$  be a $(x\bfk[x],P)$-module. Then $M$ is irreducible if and only if it is of dimension one.
\end{theorem}

\begin{proof}
It is enough to prove that every nonzero $(x\bfk[x],P)$-module $(M,p)$ has a $1$-dimensional submodule.
Due to Proposition \ref{tangtang31}, $M$ is a $n(>0)$-dimensional $\bfk$-vector space and $x, p:M\rightarrow M$ are linear maps with matrices $A$ and $B$ corresponding to an appropriate basis of $M$ respectively, satisfying $AB=-BAB$. Namely, we have $xp=-pxp$ as linear maps on $M$.

For $\alpha\in \bfk$, let
$$
M_{\alpha}(p)=\{v\in M| p(v)=\alpha v\}, \ \ M_{\alpha}(x)=\{v\in M| x v =\alpha v\}.
$$
Then $\alpha$ is an eigenvalue of $p$ (resp. $x$) if $M_{\alpha}(p)\neq 0$ (resp. $M_{\alpha}(x)\neq 0$).

{\it Case 1.} When $M_{-1}(p)\neq 0$.

 For any $v\in M_{-1}(p)$, then $p(v)=-v$. Therefore,
$$
(-1) xv=-x (-p(v))=(xp) (v)=(-pxp)(v)=-px(p(v))=p(xv),
$$
which yields that $xv \in M_{-1}(p)$.  In other words, the eigenspace $M_{-1}(p)$ is invariant under $x$. Note that $\bfk$ is algebraically closed. Hence the linear map $x|_{M_{-1}(p)}: M_{-1}(p)\rightarrow M_{-1}(p)$ has an eigenvector $u\in M_{-1}(p)$ such that $x|_{M_{-1}(p)}(u)=xu=\beta u$ for some $\beta\in \bfk$.  In view of $u\in M_{-1}(p)$ we also have $p(u)=-u$. Now let $N=\bfk u$ be a $1$-dimensional subspace of $M$. As we have seen, $x N\subseteq N$ and $p(N)\subseteq N$, then $N$ is a $1$-dimensional submodule of $M$.

{\it Case 1.} When $M_{-1}(p)= 0$.

If $M_{0}(p)= M$, then $p(v)=0$ for all $v\in M$.  Take an eigenvector $u\in M$  of linear map $x$ and let $N=\bfk u$. Similar to Case 1 we see that $N$ is a $1$-dimensional submodule of $M$. Otherwise one can find an element $\alpha\in \bfk\setminus\{-1, 0\}$ such that $M_{\alpha}(p)\neq 0$. We claim that $x M_{\alpha}(p) \subseteq M_{-1}(p)$. In fact, for any $v\in M_{\alpha}(p)$, we have $p(v)=\alpha v$, i.e., $v=\alpha^{-1}p (v)$. Therefore,
$$
p (xv)= p(x\alpha^{-1}p (v))=\alpha^{-1}(pxp)(v)=-\alpha^{-1}xp(v)=-\alpha^{-1}x(\alpha v)=(-1)xv.
$$
This proves the above claim. It follows by  $M_{-1}(p)= 0$ that $x M_{\alpha}(p)=0$. Now we let $u\in  M_{\alpha}(p)$ with $u\neq 0$ and $N=\bfk u$. Therefore, by $p(u)=\alpha u$ and $x u=0$ we see that $N$ is a $1$-dimensional submodule of $M$.  The proof is completed.
\end{proof}

Similarly, we have the same result for $(\bfk[x],P)$-module as follows.
\begin{coro}\label{tangxzhong}
Let $(M,p)$  be a $(\bfk[x],P)$-module. Then $M$ is irreducible if and only if it is of dimension one.
\end{coro}

\begin{theorem}\label{tganz}
Suppose that $(M,p)$ is an $n$-dimensional $(x\bfk[x],P)$-module such that the map $p:M\rightarrow M$ is indecomposable, i.e., the matrix of $p$ corresponding to an appropriate basis of $M$ has exactly one Jordan block. Then there is a basis  $\{\epsilon_1, \cdots, \epsilon_n\}$ of $M$ and $t_i, s_i\in \bfk, i=1, \cdots,n$ such that $x$ and $p$ act on $M$ are determined by the one of the following cases:

\begin{enumerate}
\item  For any element $v=k_1\epsilon_1+\cdots +k_n\epsilon_n\in M$,
\begin{eqnarray*}
&& x v=(k_1t_1+\cdots+k_nt_n)\epsilon_1, \\
&& p  (v) =(k_2-k_1)\epsilon_1+\cdots +(k_n-k_{n-1})\epsilon_{n-1}-k_n\epsilon_n;
\end{eqnarray*}

\item For any element $v=k_1\epsilon_1+\cdots +k_n\epsilon_n\in M$,
\begin{eqnarray*}
&& x v=k_n(s_1\epsilon_1+\cdots +s_1\epsilon_1), \\
&& p  (v) =k_2\epsilon_1+\cdots +k_n \epsilon_{n-1};
\end{eqnarray*}

\item  For any element $v=k_1\epsilon_1+\cdots +k_n\epsilon_n\in M$,
\begin{eqnarray*}
&& x v= 0, \\
&& p  (v) =(k_2+bk_1)\epsilon_1+\cdots +(k_n+bk_{n-1})\epsilon_{n-1}+bk_n\epsilon_n,
\end{eqnarray*}
where $b\in \bfk \setminus \{-1, 0\}$.
\end{enumerate}
\end{theorem}

\begin{proof}
Since $p:M\rightarrow M$ is irreducible, we assume that the matrix of $p$ corresponding to basis $\{\epsilon_1, \cdots, \epsilon_n\}$ of $M$ has one Jordan block
as
$$
J_{n}(b) = \begin{bmatrix}
b&1&\cdots&0&0 \\
0&b&\ddots&0&0 \\
\vdots&\vdots&\ddots&\ddots&\vdots \\
0&0&\cdots&b&1 \\
0&0&\cdots&0&b
\end{bmatrix}
$$
for some $b\in \bfk$. Denote by $X$  the matrix of $x$ (regard as a linear map on $M$) corresponding to basis $\{\epsilon_1, \cdots, \epsilon_n\}$. It follow by Proposition \ref{tangtang31}
that $XJ_n(b)=-J_n(b) X J_n(b)$. Corollary \ref{ttzz} tells us that
\begin{enumerate}
\item  When $b=-1$, then
$$
X=\begin{pmat}[{...}]
t_1 &\cdots& t_{n-1} & t_n\cr\-
0&\ldots& 0 & 0\cr
\vdots&\ddots&\vdots & \vdots\cr
0&\ldots & 0&0\cr
\end{pmat}
$$
for some $t_1, \cdots, t_n\in \bfk$;

\item When $b=0$, then
$$
X=\begin{pmat}[{..|.}]
0 &\cdots& 0 & s_1\cr
0&\ldots& 0 & s_2\cr
\vdots&\ddots&\vdots & \vdots \cr
0&\ldots & 0&s_n\cr
\end{pmat}
$$
for some $s_1, \cdots, s_n\in \bfk$;

\item When $b\neq -1, 0$, then $X=0$.
\end{enumerate}
For every case, the actions of $x$ and $p$  on $v=k_1\epsilon_1+\cdots +k_n\epsilon_n\in M$ are easily determined, which yield the conclusion.
\end{proof}

\begin{remark}
Let $(M,p)$  be an $(x\bfk[x],P)$-module.  If the matrix of $p$ corresponding to an appropriate basis of $M$ has exactly one Jordan block (or equivalent we say that $p$ is indecomposable), then $(M,p)$ is indecomposable. We give all such indecomposable $(x\bfk[x],P)$-module in Theorem \ref{tganz} by determined the action of $x$, which implies that the indecomposable $(x\bfk [x], P)$-module can be of any dimension. In addition, it is natural to ask whether all the indecomposable $(x\bfk[x],P)$-modules are derived from indecomposable action of $p$ with some suitable $x$ and the answer is no. For example, let $M=\bfk \epsilon_{1} \oplus \bfk \epsilon_{2}$ with
$$
x(\epsilon_{1},\epsilon_{2}) = (\epsilon_{1},\epsilon_{2})\begin{bmatrix}2&1\\0&2 \end{bmatrix}, \quad p(\epsilon_{1},\epsilon_{2}) = (\epsilon_{1},\epsilon_{2})\begin{bmatrix}-1&0\\0&0 \end{bmatrix}.
$$
The $(M,p)$ is an indecomposable $(x\bfk[x],P)$-module since $M$ is a decomposable $x\bfk[x]$-module. But it is clear that $p$ is not indecomposable.
\end{remark}

\begin{remark}
Let $M=\bfk \epsilon_{1} \oplus \bfk \epsilon_{2}\oplus \bfk \epsilon_{3}$ with
$$
x(\epsilon_{1},\epsilon_{2}, \epsilon_{3}) = (\epsilon_{1},\epsilon_{2}, \epsilon_{3})\begin{pmat}[{|.}]0&0&0\cr\- 0&0&1\cr 0&0&0\cr \end{pmat}, \quad
p(\epsilon_{1},\epsilon_{2}, \epsilon_{3}) = (\epsilon_{1},\epsilon_{2}, \epsilon_{3})\begin{pmat}[{.|.}]-1&1&0\cr 0&-1&0\cr\- 0&0&0\cr \end{pmat}.
$$
Then it is easy to see that $(M,p)$ is an indecomposable $(x\bfk[x],P)$-module, but as an $x\bfk[x]$-module it is decomposable since $M=M_1\oplus M_2$ with $x\bfk[x]$-modules $M_1=\bfk\epsilon_{1}$ and $M_2= \bfk \epsilon_{2}\oplus \bfk\epsilon_{3}$, and $p$ is not indecomposable since $M=M_1^\prime\oplus M_2^\prime$ with $p$-invariant subspaces $M_1^\prime=\bfk\epsilon_{1}\oplus \bfk\epsilon_{2}$ and $M_2^\prime= \bfk\epsilon_{3}$. More examples can be viewed in the last section.
\end{remark}

\begin{remark}
As in pointed out in \cite{G-Lin}, the category $(x\bfk[x],P)\Mod$ of $(x\bfk[x],P)$-modules is an abelian category. There is a  forgetful functor $(x\bfk[x],P)\Mod\rightarrow x\bfk[x]\Mod$ forgetting the operator $p$, which is exact and faithful. This allows us apply specific examples of the abelian category to some deep problems.
\end{remark}

\section*{Acknowledgments}
This work is supported in part by National Natural Science Foundation of China (Grant No. 11771069) and the fund of Heilongjiang Provincial Laboratory of the Theory and Computation of Complex Systems.

\end{document}